\documentclass[11pt]{amsart}
\usepackage{amssymb}
\usepackage{bm}
\usepackage{graphicx}
\usepackage[centertags]{amsmath}
\usepackage{amsfonts}
\usepackage{amsthm}
\usepackage{a4}
\usepackage{latexsym}
\usepackage{amsmath}
\usepackage{amsfonts}
\usepackage{amssymb}
\usepackage{amscd}
\usepackage{amsthm}
\usepackage{layout}
\usepackage{color}

\theoremstyle{plain}
\theoremstyle{definition}
\newtheorem{theorem}{Theorem}[section]

\newtheorem{proposition}[theorem]{Proposition}

\newtheorem*{thm*}{Theorem}

\newtheorem*{dfn*}{Definition}

\newtheorem*{cor*}{Corollary}

\newtheorem*{prp*}{Proposition}

\newtheorem*{rmk*}{Remark}

\newtheorem*{qst}{Question}

\def\XXint#1#2#3{{\setbox0=\hbox{$#1{#2#3}{\int}$ }
\vcenter{\hbox{$#2#3$ }}\kern-.6\wd0}}

\long\def\symbolfootnote[#1]#2{\begingroup%
\def\thefootnote{\fnsymbol{footnote}}\footnote[#1]{#2}\endgroup}

\begin{document}
\title{Operators on  Bourgain-Delbaen's spaces.}
\author[D. Puglisi]{Daniele Puglisi}
\address{Department of Mathematics and Computer Sciences, University of Catania,  Catania, 95125, Italy (EU)}
\email{dpuglisi@dmi.unict.it}

\date{}

\maketitle

\symbolfootnote[0]{\textit{2010 Mathematics Subject
Classification:} Primary 47B10, }


\symbolfootnote[0]{}

\symbolfootnote[0]{The author was supported by ``National Group for Algebraic and Geometric Structures, and their Applications'' (GNSAGA - INDAM).}

\begin{abstract}
We prove that whenever we pick real numbers $a,b$ such that $0<b<a<1$, $a+b>1$ and $a^3+b^3=1$, then every operator from $\ell_2$ to $X_{a,b}$ and from  $X_{a,b}$ to  $\ell_2$ must be compact, where  $X_{a,b}$ is the  Bourgain-Delbaen's space.
\end{abstract}

\section{Introduction}
This note is concerning a question raised to us by R. Aron during the conference "{\em Function Theory on Infinite Dimensional Spaces XIV}. 08 - 11 February 2016, Madrid. We were told that  the same question was raised during some other conference in 1987 and it was solved  by J. Bourgain  soon after. Since at that time J. Bourgain decided to not published it,  the result  remains only announced in \cite{AAF}. \smallskip

Throughout this note we shall use a standard notation. For Banach spaces $X,Y$ we denote by $\mathcal{L}(X, Y)$ the space of all bounded linear operators from $X$ to $Y$ and by $\mathcal{K}(X, Y)$ its ideal of compact operators. It is a classical Pitt's result \cite{Pitt} that every linear continuous operator from $\ell_p$ to $\ell_q$ must be compact for all $q<p$ .  It is worth to say here that this result has been generalized  in \cite[Lemma 3.5]{Pu} where $\ell_p$-direct sum of sequences of suitable Banach spaces have been used. Following Pitt's theorem, in 
\cite[Lemma 5]{AAF}  it was proved the folllowing
\begin{proposition}
Suppose a Banach space $X$ has the property that for some $p>1$, $\mathcal{L}(X,\ell_p) = \mathcal{K}(X,\ell_p)$. Then $\mathcal{L}(X,\ell_q) = \mathcal{K}(X,\ell_q)$ for all $1<q<p$.
\end{proposition}  
This proposition shows that there are non-trivial examples of triple $(X,Y,Z)$ of Banach spaces with the property that if every continuous linear operator from $X$ to $Y$ is compact and every continuous linear operator from $Y$ to $Z$ is compact, then every continuous linear operator from $X$ to $Z$ is compact. In case this happens, we say that the triple $(X,Y,Z)$  has the {\em compact operators transitivity property}. Therefore, it is quite natural to raise the following
\begin{qst}[R. Aron, 1987]
Does every triple $(X,Y,Z)$ of Banach spaces has the compact operators transitivity property?
\end{qst}
As announced in \cite{AAF}, the answer was given in negative by J. Bourgain and in this note we propose a proof of such result. Another negative solution of the above question was given in \cite{AH}, where, if we denote by $\mathfrak{X}_{AH}(L)$ the  Argyros-Haydon  space constructed on the infinite subset $L \subseteq \mathbb{N}$,
the following spectacular result holds
\begin{theorem}
If $L \cap L^\prime$ is finite, then every bounded linear operator from  $\mathfrak{X}_{AH}(L)$ to  $\mathfrak{X}_{AH}(L^\prime)$ must be compact.
\end{theorem}

\section{The Bourgain-Delbaen's space}

In this section we would like to review briefly the construction of the classical Bourgain-Delbaen's space $X_{a,b}$  using the original notation (see \cite{BD}), even if we recommend the reader the Ph.D. Thesis of  M. Tarbard \cite{T}, where the  exposition is very clear. \smallskip

For a  $\lambda>1$ let us fix from now on real numbers $a, b$ such that
\begin{enumerate}
\item[($1$)] $0<b<a<1$,
\item[($2$)] $a+2b\lambda \leq \lambda$,
\item[($3$)] $a+b >1$.
\end{enumerate} 

By induction, first we construct  a sequence $(d_n)_n \subseteq \mathbb{R}$ and 
$$i_{m, l}: \hbox{span}\{e_i: \ 1 \leq i \leq d_m\} \longrightarrow   \hbox{span}\{e_i: \ 1 \leq i \leq d_l\} $$
with $m<l \leq n$, where $e_i$ is the standard basic sequence, i.e.  $e_i(j)=\delta_{i,j}$, as follows.\smallskip

Let $d_1=1$. Suppose $d_1, \dots, d_n$ are known as well as $\{i_{m, l}: \ 1 \leq m < l \leq n\}$ satisfying 
\begin{enumerate}
\item[($a$)] $\pi_m \circ i_{m, n} = \hbox{id}_{\hbox{span}\{e_i: \ 1 \leq i \leq d_m\}}$ for $m < n$,
\item[($b$)] $i_{l, n} \circ i_{m, l} = i_{m, n}$, for $m < l < n$,
\end{enumerate}
where $\pi_m$ is the natural projection to the first $d_m$ coordinates. For all 
\begin{align*}
\gamma &= (m, i, j, \varepsilon^\prime, \varepsilon^{\prime \prime}), \ \hbox{with} \  m<n,\\
&  1 \leq i < d_m, \ 1 \leq j \leq  d_n \ \hbox{and} \  \varepsilon^\prime, \varepsilon^{\prime \prime} = \pm 1
\end{align*}
let us define a functional in $(\hbox{span}\{e_i: \ 1 \leq i \leq d_n\})^*$ as
$$c_\gamma^* (x) =  \varepsilon^\prime a x_i + \varepsilon^{\prime \prime} b (x - \pi_m \circ i_{m, n} (x))_j.$$
Let us consider the set of all such functionals, namely 
\begin{align*}
\mathcal{F}_n= \{c_\gamma^*: \ \gamma = (m, i, j, \varepsilon^\prime, \varepsilon^{\prime \prime}),\ m<n, \  1 \leq i < d_m, \ 1 \leq j \leq  d_n \ \hbox{and} \  \varepsilon^\prime, \varepsilon^{\prime \prime} = \pm 1\};
\end{align*}
then, we define
$$d_{n+1} = d_n + \ \hbox{card}\mathcal{F}_n$$
and
$$i_{n,n+1}(x) := ( \underbrace{x_1, \dots, x_{d_n}, (c_\gamma^*(x))_{\gamma \in \mathcal{F}_n}}_{d_{n+1}}, 0, 0, \dots, 0, \dots).$$
For each $n \in \mathbb{N}$ let us  consider
$$i_n :  \hbox{span}\{e_i: \ 1 \leq i \leq d_n\} \longrightarrow   \ell_\infty$$
defined as
$$i_n = \lim_{m \rightarrow \infty} (i_{n, n+1} \circ  i_{n+1, n+2} \circ \dots \circ  i_{m-1, m}). $$
It has been shown by J. Bourgain and F. Delbaen \cite{BD} that $\|i_n\| \leq \lambda$ for all $n \in \mathbb{N}$. Therefore, if we denote by
$$F_n = i_n( \hbox{span}\{e_i: \ 1 \leq i \leq d_n\}) \subseteq \ell_\infty,$$
the space of Bourgain-Delbaen $X_{a,b}$ is defined as the closure in $\ell_\infty$ of the countable union
$$\bigcup_{n \in \mathbb{N}} F_n.$$ 
By construction, it follows that the Banach-Mazur distance $d(F_n, \ell_\infty^{d_n}) \leq \lambda$ for each $n \in \mathbb{N}$. Thus, 
\begin{enumerate}
\item[($\alpha$)]  $X_{a,b}$ is a $\mathcal{L}_\infty$-space. 
\end{enumerate}
Curiously, the space $X_{a,b}$  satisfies the following properties, which make the space very interesting:
\begin{enumerate}
\item[($\beta$)]$X_{a,b}$ does not contain any copy of $\ell_1$, which implies that $X_{a,b}^*$ is isomorphic to  $\ell_1$;
\item[($\gamma$)]  $X_{a,b}$ has the Radon-Nikodym property;
\item[($\delta$)]  $X_{a,b}$ is reflexive saturated (i.e., every infinite dimensional closed subspace contains an infinite dimensional  reflexive subspace).
\end{enumerate}
Before to close this section, we would like to review a last peculiar property. Let $\alpha$ be the unique number such that
$$a^{\frac{1}{1-\alpha}} +b^{\frac{1}{1-\alpha}} =1.$$
Then, in $X_{a,b}$ the following holds. 
\begin{proposition}\label{main BD}
For every sequence $(x_n)_n \subseteq X_{a, b}$ such that 
\begin{enumerate}
\item[($i$)] $\|x_n\|=1$ for all $n \in \mathbb{N}$,
\item[($ii$)] $x_n \longrightarrow 0$ weakly,
\end{enumerate}
there exist $C>0$ and a subsequence $(x_{k_n})_n$ such that 
\begin{equation}\label{main ineq}
\|x_{k_1} + \dots + x_{k_n}\| \geq C n^{\alpha}, \quad \forall n \in \mathbb{N}.
\end{equation}
\end{proposition}

\section{Operators on the space $X_{a,b}$}
Throughout this section we fix  $a,b$ such that 
\begin{enumerate}
\item[($i$)]$0<b<a<1$, 
\item[($ii$)]$a+b>1$,  
\item[($iii$)]$a^3+b^3=1$.
\end{enumerate}
Before to state the main result, we would like to recall a well known result. Since the proof is easy, we insert it for the sake of completeness. 
\begin{proposition}
Let $X$ be a Banach space with a Schauder basis, $Y$ any Banach space and let $T:X \longrightarrow Y$ be a bounded linear operator. If 
\begin{align}\label{1}
\|T(x_n)\| \stackrel{n \rightarrow \infty}{\longrightarrow} 0 \ \hbox{for all bounded block basic sequence} \ (x_n)_n,
\end{align}
then $T$ is compact.
\end{proposition}

\begin{proof}
Let us denote by $(x_n)_n$ a Schauder basis of $X$ with basis constant $K$, and let us denote by
$$P_n: X \longrightarrow X$$
be the $n$-th projection associate to the basis; i.e., 
$$P_n(\sum_{i=1}^\infty a_i x_i) =   \sum_{i=1}^n a_i x_i.$$
We show that
$$\|T - T\circ P_n\| \stackrel{n \rightarrow \infty}{\longrightarrow} 0;$$
i.e., $T$ is uniform limit of finite rank operators.\smallskip

Let us assume we are not in this situation, that means there exist $\delta>0$, a strictly increasing sequence of natural number $(k_n)_n$ and a sequence of norm-one $(z_n)_n \subseteq X$ such that
$$ \|(T - T\circ P_{k_n})(z_n)\| > \delta.$$ 
Since $z_1 = \lim_n P_n(z_1)$, we can find $s_1 > k_1$ such that
$$\|z_1 - P_{s_1}(z_1)\| < \frac{\delta}{2 \|T\|}.$$
Let us define $y_1 = P_{s_1}(z_1) -  P_{k_1}(z_1)$. Then
\begin{align*}
\|y_1\| &\leq 2 K \ \hbox{and}\\
\|T(y_1)\| &\geq \|(T - T\circ P_{k_1})(z_1)\| - \|(T - T\circ P_{s_1})(z_1)\|\\
& > \frac{\delta}{2}
\end{align*}
Let $k_{j_2} > s_1$. Similarly, we find $s_2 >  k_{j_2}$ such that
$$\|z_{j_2} - P_{s_2}(z_{j_2})\| < \frac{\delta}{2 \|T\|}.$$
Let us define $y_2 = P_{s_2}(z_{j_2} ) -  P_{k_{j_2} }(z_{j_2} )$. Then
\begin{align*}
\|y_2\| &\leq 2 K \ \hbox{and}\\
\|T(y_2)\| & > \frac{\delta}{2}
\end{align*}
Iterating this process we find a bounded block basic sequence $(y_n)_n$ of $Y$ such that $\|T(y_n)\| >\frac{\delta}{2}$, against the assumption (\ref{1}).
\end{proof}

\begin{theorem}
Every bounded linear operator from $\ell_2$ to $X_{a,b}$ and from  $X_{a,b}$ to  $\ell_2$ must be compact.
\end{theorem}
\begin{proof}
Since $X_{a,b}^*$  has the Schur property, it is obvious that any operator $X_{a,b}$ to  $\ell_2$ is compact.\smallskip

Suppose  there is  a bounded linear operator 
$$T: \ell_2 \longrightarrow X_{a,b}$$
which is not compact. By the previous proposition, there must exists a  bounded block basic sequence $(x_n)_n \subseteq \ell_2$ such that 
$$\|T(x_n)\| > \delta, \ \hbox{for some} \ \delta>0 .$$
It is worth  to recall that a basis in a Banach space is shrinking if and only if  any bounded block basic sequence is weakly null. Since the basis in $\ell_2$ is shrinking, from one hand we have  that the sequence $(x_n)_n$ is weakly null (and so it is $(T(x_n))_n$). Therefore, the sequence  
$$\left(\frac{T(x_n)}{\|T(x_n)\|}\right)_n \ \hbox{is weakly null.}$$
By Proposition \ref{main BD} and assumption ($iii$) above, unless to pass through a subsequence, there exists $C>0$ such that
$$\left\| \frac{T(x_1)}{\|T(x_1)\|} + \dots + \frac{T(x_n)}{\|T(x_n)\|} \right\| \geq C n^{\frac{2}{3}},$$

On the other hand, since bounded block basic sequence in $\ell_2$ are unconditionally basic sequence (and so it is $(T(x_n))_n$),  it follows also that 
$$\| T(x_1) + \dots + T(x_n) \| \geq C^\prime n^{\frac{2}{3}},$$ 
for some $C^\prime>0$. Therefore,
\begin{align*}
C^\prime n^{\frac{2}{3}} &\leq \| T(x_1) + \dots + T(x_n) \|\\
& \leq \|T\| \|x_1 + \dots + x_n\|_{\ell_2}\\
&\leq \|T\| C^{\prime \prime} n^{\frac12} ,
\end{align*}
where the last inequality follows from the fact that $(x_n)_n$ is a bounded block sequence in $\ell_2$. This would imply that
$$n^{\frac16} \leq \|T\|  \frac{C^{\prime \prime}}{C^\prime}.$$
Namely a contradiction. This completes the proof.
\end{proof}
We would like to finish this note with the following natural questions.
\begin{qst}
\begin{enumerate}
\item[($i$)] Does there exists a pair of Banach spaces $(X, Y)$ such that every $n$-homogeneous polynomial  from $X$ to $Y$ and from  $Y$ to  $X$ must be compact?
\item[($ii$)]  Let $Y$ and $Z$ be reflexive Banach spaces. Does, for every Banach space $X$, the triple $(X, Y, Z)$ has the compact operators transitivity property?
\end{enumerate}
\end{qst}

\subsection*{Acknowledgement} We wish to thank R. Aron for bringing to our attention the question.

\end{document}